\documentclass{article}
\usepackage[utf8]{inputenc}
\usepackage{amsmath}
\usepackage{amssymb}
\usepackage{amsthm}
\usepackage{xcolor}
\usepackage{fullpage}
\usepackage{hyperref}
\usepackage{tikz}
\usepackage{soul}
\usepackage{mathrsfs}
\usepackage{enumitem}

\newtheorem{theorem}{Theorem}[section]
\newtheorem{Lemma}[theorem]{Lemma}
\newtheorem{corollary}[theorem]{Corollary}
\newtheorem{lemma}[theorem]{Lemma}
\newtheorem{definition}[theorem]{Definition}
\newtheorem{proposition}[theorem]{Proposition}
\newtheorem{remark}[theorem]{Remark}
\newtheorem{example}[theorem]{Example}

\newtheorem*{theorem*}{Theorem}

\title{Spectral multipliers for maximally subelliptic operators}
\author{Lingxiao Zhang}

\begin{document}
\maketitle

\begin{abstract}
Consider a non-negative, self-adjoint, maximally subelliptic operator on a compact manifold. We show that the spectral multiplier is a singular integral operator under an appropriate Mihlin-H\"ormander type condition. We establish the equivalence between non-isotropic Besov and Triebel-Lizorkin spaces adapted to the operator and those adapted to a Carnot-Carath\'eodory geometry on the manifold. We also give a Mihlin-H\"ormander type condition for the boundedness of the spectral multiplier on non-isotropic $L^p$ Sobolev spaces. 
\end{abstract}

\section{Introduction}
Let $X$ be a connected compact $C^\infty$ manifold endowed with a strictly positive smooth density $\mu$. Consider $(W,d):=\{(W_1, d_1), \ldots, (W_\nu, d_\nu)\}$ consisting of finitely many smooth vector fields $W_i$ on $X$ each paired with a formal degree $d_i \in \mathbb{N}_+=\{1,2,\ldots\}$. Let $W_1, \ldots, W_\nu$ satisfy H\"ormander's condition, i.e., the Lie algebra generated by these vector fields spans the tangent space at every point in $X$. By Nagel, Stein, and Wainger \cite{NSW85}, $(W,d)$ gives rise to Carnot-Carath\'eodory balls and a Carnot-Carath\'eodory metric on $X$ as follows.

\begin{definition}
The Carnot-Carath\'eodory ball is defined as the set
\begin{align*}
B(x,\delta) :=\Big\{ y\in X \Big| & \exists \gamma:[0,1] \to X, \gamma(0) = x, \gamma(1) = y,\\
& \gamma \text{ is absolutely continuous},\\
& \gamma'(t) = \sum_{i=1}^\nu a_i(t) \delta^{d_i} W_i(\gamma(t)) \text{ almost everywhere},\\
& a_i\in L^\infty[0,1], \Big\| \Big(\sum_{i=1}^\nu |a_i|^2 \Big)^{\frac{1}{2}}\Big\|_{L^\infty[0,1]}<1 \Big\}.
\end{align*}
\end{definition}

\begin{definition}
The Carnot-Carath\'eodory metric is defined as
$$
\rho(x,y)= \inf \{\delta>0: y\in B(x,\delta)\}.
$$
\end{definition}

It is a result of Chow \cite{chow} that $\rho$ is a metric; in particular, $\rho(x,y)<\infty, \forall x,y\in X$. Nagel, Stein, and Wainger \cite{NSW85} showed $(X,\rho, \mu)$ is a doubling space: $\exists Q, Q'>0$ such that
\begin{equation}\label{doubling}
C^{Q'} \mu(B(x,\delta)) \lesssim \mu(B(x,C\delta)) \lesssim C^Q \mu(B(x,\delta)), \quad \forall C\geq 1, \delta>0, x\in X.
\end{equation}

For any $M\in \mathbb{N}$, and for $\alpha =(\alpha_1, \ldots, \alpha_M)\in \{1, \ldots, \nu\}^M$ a list of elements of $\{1,\ldots, \nu\}$, we set $W^\alpha = W_{\alpha_1}\cdots W_{\alpha_M}$, $|\alpha|=M$, and $\text{deg}\,\alpha = d_{\alpha_1} + \cdots + d_{\alpha_M}$. We call such $\alpha$ an ordered multi-index.

Let $\kappa \in \mathbb{N}_+$ be such that $h_i:= \kappa/d_i \in \mathbb{N}_+$ for every $i$. Consider a non-negative, symmetric, degree $2\kappa$ operator acting on $L^2(X,\mu)$ with dense domain $C^\infty(X)$ of the form
\begin{equation}\label{form}
L_0= \sum_{\text{deg}\,\alpha \leq \kappa, \text{deg}\,\beta \leq \kappa} b_{\alpha, \beta}(x) W^\alpha W^\beta, \quad b_{\alpha, \beta} \in C^\infty(X).
\end{equation}
Note in the special case that all formal degrees $d_i =1$, $L_0$ just needs to be non-negative and symmetric of the form
$$
\sum_{|\alpha|\leq 2\kappa} b_\alpha(x) W^\alpha, \quad b_\alpha \in C^\infty(X).
$$
Suppose $L_0$ is maximally\footnote{Maximal subellipticity, also called maximal hypoellipticity, is a far-reaching generalization of ellipticity, first introduced in the nontrivial case by Folland and Stein \cite{FOLLAND}. It covers a much larger class of PDE than elliptic PDE.
There is a different definition for the maximal subellipticity:
$$
\sum_i \|W_i^{2h_i} f\|_{L^2(X)}^2 \lesssim \|L_0f\|_{L^2(X)}^2 + \|f\|_{L^2(X)}^2, \quad \forall f\in C^\infty(X),
$$
which is equivalent to the above definition (\ref{maxsub}) in this setting by Chapter 8 of Street \cite{BookII}.} subelliptic:
\begin{equation}\label{maxsub}
\sum_{i=1}^\nu \|W_i^{h_i} f\|_{L^2(X)}^2 \lesssim \langle L_0f, f \rangle_{L^2(X)} + \|f\|_{L^2(X)}^2, \quad \forall f\in C^\infty (X).   
\end{equation}
Let $L$ be any non-negative, self-adjoint extension of $L_0$. The existence of such an extension is guaranteed by the Friedrichs extension.

For every bounded Borel measurable function $m: [0, \infty) \to \mathbb{C}$, the spectral multiplier $m(L)$ is defined by the spectral theorem. By Schwartz kernel theorem, we can identify each continuous linear operator $T: C^\infty(X) \to C^\infty(X)'$ with its Schwartz kernel in $C^\infty(X\times X)'$, denoted by $T(x,y)$. $T$ is continuous from $C^\infty(X)'$ to $C^\infty(X)$ if and only if $T(x,y)\in C^\infty(X\times X)$.

We will define singular integral operators of order $t\in \mathbb{R}$ in Definition \ref{nis}, which form a filtered algebra\footnote{Singular integral operators of order $t$ were first introduced by Nagel, Rosay, Stein, and Wainger \cite{NRSW89} under the name NIS operators of order $-t$. NIS operators were later studied by Chang, Nagel, and Stein \cite{CNS92}, Koenig \cite{koenig}. See Section 2.16 of Street \cite{BOOKI} for a history. For a proof that they form a filtered algebra, see Street \cite{BookII}.}: if $T_1, T_2$ are singular integral operators of orders $t_1, t_2$, respectively, then $T_1\circ T_2$ is a singular integral operator of order $t_1+t_2$. In particular, we say $T$ is a singular integral operator of order $0$ if 
\begin{itemize}
\item
(Differential inequalities): $\forall$ordered multi-indices $\alpha, \beta$,
$$
\big| W_x^\alpha W_y^\beta T(x,y)| \lesssim_{\alpha, \beta} \frac{\rho(x,y)^{- \text{deg}\,\alpha - \text{deg}\,\beta}}{\mu(B(x, \rho(x,y)))},
$$
where $W_x$ denotes the list of vector fields $W_1, \ldots, W_\nu$ acting as partial differential operators in the $x$ variables, and similarly for $W_y$.
\item
(Cancellation conditions): See Definition \ref{nis}.
\end{itemize}

Our first result is
\begin{theorem}\label{NIS}
Let $t\in \mathbb{R}$. $m(L)$ is a singular integral operator of order $2\kappa t$ if
$$
\sup_{\lambda>0} \lambda^{-t} \big| (\lambda \partial_\lambda)^k m(\lambda) \big|<\infty, \quad \forall k\in \mathbb{N}.
$$
\end{theorem}

Street \cite{boxb} proves a similar result when $L$ is the Kohn Laplacian $\square_b$ on the boundary of a weakly pseudoconvex domain of finite type in $\mathbb{C}^2$, where $\square_b$ can have an infinite-dimensional null space. \cite{boxb} uses the finite propagation speed of the wave equation, which does not apply to our operator $L$ when the degree is greater than $2$. For other closely related results, see Nagel and Stein \cite{NS01B} and references therein.

Our second result is regarding the non-isotropic $L^p$ Sobolev spaces: $NL^p_t(X)$, where $1<p<\infty, t\in \mathbb{R}$. There are two standard ways to define the spaces, one in terms of $L$ (Definition \ref{simpleversion}), the other in terms of the Carnot-Carath\'eodory geometry (Definition \ref{triebel}). We will show they are equivalent (Theorem \ref{sobolev-equivalence}). 

\begin{definition}\label{simpleversion}
Let $\phi_0\in C_c^\infty(\mathbb{R})$ and let $\{\phi_j\}_{j>0}$ be bounded in $C_c^\infty(0,\infty)$, such that
$$
\sum_{j\in \mathbb{N}} \phi_j\big(2^{-2\kappa j} \cdot\big) \equiv 1, \quad \text{on } [0,\infty).
$$
(One can construct $\{\phi_j\}_{j\in \mathbb{N}}$ as in (\ref{phi}).)
For $t\in \mathbb{R}, 1<p<\infty$, we define the non-isotropic $L^p$ Sobolev space of order $t$ on $X$ adapted to $L$: $NL^p_t(X, L)$ to be the space of all distributions $f\in C^\infty(X)'$ with the following norm finite:
$$
\Big\| \Big\{ 2^{jt} \phi_j(2^{-2\kappa j} L) f \Big\}_{j\in \mathbb{N}} \Big\|_{L^p(X,l^2(\mathbb{N}))}.
$$
\end{definition}

Note $\phi_j(2^{-2\kappa j}L)f\in C^\infty(X)$ for $f\in C^\infty(X)'$ by Lemma \ref{etaiselementary}.

Basic properties and functional calculus results have been established on $NL^p_t(X,L)$ spaces or its variants, see e.g. Bui and Duong \cite{TRIEBEL} and Georgiadis, Kerkyacharian, Kyriazis and Petrushev \cite{GKKP17}. We will show $NL^p_t(X,L)$ does not depend on the choice of $L$, by proving that $NL^p_t(X,L)$ is equivalent to $NL_t^p(X,(W,d))$ defined in terms of the Carnot-Carath\'eodory geometry (see Definition \ref{triebel}).

In particular, in the case with all formal degrees $d_i=1$ and $t\in \mathbb{N}$, $NL^p_t(X,(W,d))$ consists of all $f\in L^p(X)$ with
\begin{equation}\label{example}
W^\alpha f \in L^p(X), \quad \forall \alpha \text{ with } |\alpha|\leq t.
\end{equation}

\begin{example}
Consider the case\footnote{While $\mathbb{R}^2$ is not compact, the same ideas hold in this case as well.} $X = \mathbb{R}^2$ with $(W, d) =\{(\partial_x,1), (x\partial_y,1)\}$. The vector fields satisfy H\"ormander's condition, since $[\partial_x, x\partial_y] = \partial_y$. For $t\in \mathbb{N}$,
$$
NL^p_{2t}(X,(W,d)) = \Big\{f\in L^p(X): \partial_x^{2t}f,  (x\partial_y)^{2t} f\in L^p(X) \Big\}= \Big\{f\in L^p(X): \partial_x^{2t}f, \partial_y^{t}f, (x\partial_y)^{2t} f\in L^p(X) \Big\},
$$
where the latter equality follows from Rothschild and Stein \cite{rothschild}. Thus in this case, we have
$$
L^{p}_{2t}(X) \subseteq NL^{p}_{2t}(X,(W,d)) \subseteq L^{p}_{t}(X), \quad \forall t\in \mathbb{N}.    
$$
For the general theory about relations between $NL^p_t(X,(W,d))$ and the standard $L^p$ Sobolev spaces, see Street \cite{BookII}.
\end{example}

Our second result is

\begin{theorem}\label{sobolev-equivalence}
For $t\in \mathbb{R}, 1<p<\infty$, the non-isotropic $L^p$ Sobolev space of order $t$ on $X$ adapted to $L$: $NL^p_t(X,L)$ is equal to the non-isotropic $L^p$ Sobolev space of order $t$ on $X$ adapted to the Carnot-Carath\'eodory geometry (See (\ref{example}) and Definition \ref{triebel}): $NL^p_t(X,(W,d))$.
\end{theorem}

Since $NL_t^p(X,(W,d))$ is independent of the choice of $L$, we immediately have

\begin{corollary}
Let $t\in \mathbb{R}$ and $1<p<\infty$. Given $X$ and $(W,d)$, the space $NL^p_t(X, L)$ does not depend on the choice of the maximally subelliptic, non-negative, self-adjoint operator $L$.
\end{corollary}

The equivalence also holds for the more general non-isotropic Besov and Triebel-Lizorkin spaces, see Theorem \ref{full-equivalence} and Corollary \ref{indep}. 

For a similar result, see Street \cite{BookII}, where such equivalence is established in the case when all the $d_i =1$ and when $L=W_1^*W_1 + \cdots + W_\nu^*W_\nu$, H\"ormander's sub-Laplacian.

Our third result is regarding the boundedness of $m(L)$ on the non-isotropic $L^p$ Sobolev spaces. Since $NL_t^p(X,L)=NL_t^p(X,(W,d))$, we will denote the non-isotropic $L^p$ Sobolev spaces by $NL_t^p(X)$, where $1<p<\infty, t\in \mathbb{R}$. The $m(L)$ satisfying the condition of Theorem \ref{NIS} is bounded on all $NL^p_t(X)$ spaces, due to Street \cite{BookII} (see Proposition \ref{nisbound} in this paper). But $m$ does not need to be infinitely smooth for $m(L)$ to be bounded on these spaces. We will quantify the amount of smoothness required for $m(L)$ to be bounded on each $NL^p_t(X)$ space. For that, we define the following notion.

\begin{definition}
Fix a nonzero $\phi\in C_c^\infty(0,\infty)$. For every $1<r\leq\infty, 0<s < \infty$, define a norm 
$$
\|m\|_{L^r_{s, \text{sloc}}} := \sup_{R>0} \big\|\phi(\cdot) m\big(R\cdot \big) \big\|_{L^r_s(\mathbb{R})},
$$
where $\|\cdot\|_{L^r_s(\mathbb{R})}$ denotes the standard $L^r$ Sobolev norm of order $s$. More precisely,
$$
\|f\|_{L^r_s(\mathbb{R})} = \big\| \mathcal{F}^{-1} \big( (1+4\pi^2 |\xi|^2)^{s/2} \hat f(\xi) \big) \big\|_{L^r(\mathbb{R})}, \quad \mathcal{F}^{-1} \text{ denotes the inverse Fourier transform}.
$$
\end{definition}

\begin{remark}
The conclusions in the paper do not depend on the choice of the above nonzero $\phi \in C_c^\infty(0,\infty)$. In fact, for $1<r<\infty$, by Christ \cite{CHRIST2}, whether the norm $\|\cdot \|_{L^r_{s, \text{sloc}}}$ is finite does not depend on the choice of nonzero $\phi \in C_c^\infty(0,\infty)$. For $r=\infty$, $\forall 0\neq \phi, \psi \in C_c^\infty(0,\infty), \forall \epsilon>0$,
$$
\sup_{R>0} \big\|\phi(\cdot) m\big(R\cdot \big) \big\|_{L^\infty_s(\mathbb{R})} \lesssim_{\phi, \psi,\epsilon} \sup_{R>0} \big\|\psi(\cdot) m\big(R\cdot \big) \big\|_{L^\infty_{s+\epsilon}(\mathbb{R})},
$$
and the proofs in the paper allow changing from $\|\cdot \|_{L^\infty_s}$ to $\|\cdot \|_{L^\infty_{s+\epsilon}}$ or $\|\cdot \|_{L^\infty_{s-\epsilon}}$ for sufficiently small $\epsilon>0$.

For consistency, we will fix $\phi$ as defined in (\ref{phi}).    
\end{remark}

For $t\in \mathbb{R}$, denote $\Lambda^t = \max\{\lambda^t, (1+\lambda)^t\}$, as a function of $\lambda$ on $(0,\infty)$. Recall $Q$ is defined in (\ref{doubling}). By combining the framework of Street \cite{BookII} and the result of Duong, Ouhabaz, and Sikora \cite{DUONG} on spectral multipliers, we will obtain our third result:

\begin{theorem}\label{unconditional}
Let $t, \tau\in \mathbb{R}, 1<p<\infty$. $m(L)$ is bounded from $NL^p_{2\kappa \tau}(X)$ to $NL^p_{2\kappa \tau-2\kappa t}(X)$, if 
$$
s >  Q\Big|\frac{1}{p}-\frac{1}{2}\Big|, \quad 0\leq \frac{1}{r} < \min \Big\{1, s- Q\Big|\frac{1}{p}-\frac{1}{2}\Big|\Big\},
$$
and
\begin{equation}\label{finiteness}
\Big\| \frac{m(\lambda)}{\Lambda^t} \Big\|_{L^r_{s, \text{sloc}}}<\infty.  
\end{equation}
Moreover for such $r,s$,
$$
\big\|m(L) \big\|_{NL^p_{2\kappa \tau}(X) \to NL^p_{2\kappa \tau-2\kappa t}(X)} \lesssim_{t, \tau, p, s,r} |m(0)| + \Big\| \frac{m(\lambda)}{\Lambda^t} \Big\|_{L^r_{s, \text{sloc}}}.
$$
In particular,
\begin{equation}\label{threshold2}
\big\|m(L) \big\|_{NL^p_{2\kappa \tau}(X) \to NL^p_{2\kappa \tau-2\kappa t}(X)} \lesssim_{t, \tau, p} |m(0)| + \Big\| \frac{m(\lambda)}{\Lambda^t} \Big\|_{L^2_{\frac{Q+1}{2}, \text{sloc}}},    
\end{equation}
$$
\big\|m(L) \big\|_{NL^p_{2\kappa \tau}(X) \to NL^p_{2\kappa \tau-2\kappa t}(X)} \lesssim_{t, \tau, p} |m(0)| + \Big\| \frac{m(\lambda)}{\Lambda^t} \Big\|_{L^\infty_{\frac{Q}{2}, \text{sloc}}}.
$$
\end{theorem}

\begin{remark}\label{1.8}
By the Sobolev embedding theorem, for $m$ satisfying (\ref{finiteness}), the restriction $m|_{(0,\infty)}$ is equal to a continuous function. $m(L)$ is defined using this continuous version of $m|_{(0,\infty)}$. That is, we are not allowed to change $m$ on a set of measure $0$.
\end{remark}

\begin{remark}
The range $s>Q\big|\frac{1}{p} - \frac{1}{2}\big|$ is sharp in the case when $L$ is the Laplace operator on $\mathbb{R}^Q$ with each $W_i = \partial_{x_i}, d_i =1$. See Grafakos, He, Honzik, and Nguyen \cite{grafakos} for a counterexample below the critical line $s=Q\big|\frac{1}{p} - \frac{1}{2}\big|$, and see Slav\'{\i}kov\'{a} \cite{lenka} for a counterexample on the critical line $s=Q\big|\frac{1}{p} - \frac{1}{2}\big|$.
\end{remark}

The same problem has been investigated in a myriad of other settings, especially when $L$ is some version of the Laplacian on a manifold, and there is a huge literature devoted to showing that $L$ has a differentiable functional calculus on $L^p(X)$ under suitable assumptions on $L$ and the underlying manifold (e.g., the doubling condition, Gaussian-type heat kernel bounds, finite propagation speed for the wave equation, etc). The framework of Street \cite{BookII} allows us to treat in a simple and uniform way the broad class of maximally subelliptic operators on a compact manifold.

Our third result continues in the line of the work by Christ \cite{CHRIST2} which showed the $L^p$ boundedness of $m(L)$ when $L$ is a sub-Laplacian on a graded nilpotent group (also proved independently by Mauceri and Meda \cite{MM90}) and the work by Duong, Ouhabaz, and Sikora \cite{DUONG} which showed the $L^p$ boundedness of $m(L)$ under the condition of a Gaussian bound instead of the maximal subellipticity.

As in (\ref{threshold2}), our $m$ needs to be in $L^2_{(Q+1)/2, \text{sloc}}$ for $m(L)$ to be bounded on $L^p(X)$. This is different from the condition required in the special case when $L$ is a sub-Laplacian on a graded nilpotent group (Christ \cite{CHRIST2}), where $m$ only needs to be in $L^2_{s, \text{sloc}}$ for $s>Q/2$. In fact, for the general $L$, Duong, Ouhabaz, and Sikora \cite{DUONG} gave counter-examples that $m\in L^2_s$ for $s>Q/2$ is not sufficient to make $m(L)$ bounded on $L^p(X)$.

In Section \ref{defn}, we will define the singular integral operators, and give two standard definitions for the non-isotropic Besov and Triebel-Lizorkin spaces. In Section \ref{keylemma}, we will prove a key lemma involving a set of spectral multipliers, which is used in proving Theorems \ref{NIS} and \ref{full-equivalence}. In Section \ref{sectionnis}, we will prove Theorem \ref{NIS}: $m(L)$ is a singular integral operator under the conditions of that theorem. In Section \ref{spaceequiv}, we will prove that the two standard ways to define non-isotropic Besov and Triebel-Lizorkin spaces are equivalent (Theorem \ref{full-equivalence}), from which Theorem \ref{sobolev-equivalence} follows. In Section \ref{sectionbound}, we will prove the boundedness of $m(L)$ on non-isotropic $L^p$ Sobolev spaces under a Mihlin-H\"ormander type condition (Theorem \ref{unconditional}), which uses an interpolation result in Appendix \ref{B}.

\textbf{Acknowledgements.} The author would like to thank her advisor Brian Street for introducing this problem, for his guidance and illuminating discussions. The author would also like to thank Loukas Grafakos who informed her the Calder\'on-Torchinsky multiplier theorem. The author was partially supported by the NSF grant DMS-2153069. The author would like to thank the anonymous referee for their comments.

\section{Definitions}\label{defn}
In order to define the singular integral operators, we need the following notion.

\begin{definition}[Street \cite{BookII}]
A set $\mathcal{B} \subseteq C^\infty(X) \times X \times (0,1]$ is a bounded set of bump functions if
\begin{itemize}
\item
$\forall (\eta, x, 2^{-j}) \in \mathcal{B}$, $\eta$ is supported in $B(x, 2^{-j})$,

\item
$\forall$ordered multi-index $\alpha$, $\forall (\eta, x, 2^{-j}) \in \mathcal{B}$,
$$
\sup_{y\in X} 2^{-j\,\text{deg}\,\alpha} |W^\alpha \eta(y)| \lesssim_\alpha \mu(B(x,2^{-j}))^{-1}.
$$
\end{itemize}
\end{definition}

Recall that $Q, Q'$ are defined in (\ref{doubling}).
\begin{definition}[\cite{BookII}]\label{nis}Let $t\in \mathbb{R}$.
A continuous operator $T: C^\infty(X) \to C^\infty(X)'$ is a singular integral operator of order $t$ if 
\begin{itemize}
\item 
(Differential inequalities): $T(x,y)$ is $C^\infty$ for $x\neq y$ and satisfies that for all ordered multi-indices $\alpha, \beta$, and for all $M\in \mathbb{N}$ with $t+\text{deg}\,\alpha + \text{deg}\,\beta + M >-Q'$,
\begin{equation}\label{sym}
\big| W_x^\alpha W_y^\beta T(x,y)| \lesssim_{\alpha, \beta, M} \frac{\rho(x,y)^{-(t+ \text{deg}\,\alpha + \text{deg}\,\beta + M)}}{\mu(B(x, \rho(x,y)))},
\end{equation}
\item
(Cancellation conditions): For all bounded sets of bump functions on $X$: $\mathcal{B} \subseteq C^\infty(X) \times X \times (0,1]$, for all ordered multi-index $\alpha$, and for all $M \in \mathbb{N}$ with $t+ \text{deg}\,\alpha + M > -Q'$,
$$
\sup_{\substack{(\eta, y, 2^{-j})\in \mathcal{B}\\ x\in X}} 2^{-j(t+\text{deg}\,\alpha + M)} \mu(B(x, 2^{-j})) |W^\alpha T\eta(x) | \lesssim_{\mathcal{B}, \alpha, M} 1,
$$
with the same estimates for $T^*$ in place of $T$.
\end{itemize}
\end{definition}

The doubling condition (\ref{doubling}) of the space $(X,\rho, \mu)$ immediately implies that $\forall x\in X, \forall r, R>0$,
\begin{equation}\label{doub}
\begin{aligned}
&\mu(B(x,R)) \lesssim \Big(1+\frac{R}{r}\Big)^Q \mu(B(x,r)), \\
&\mu(B(x,r))\lesssim \Big(1+ \frac{\rho(x,y)}{r}\Big)^Q \mu(B(y,r)).
\end{aligned}
\end{equation}
Thus the inequality (\ref{sym}) is essentially symmetric with respect to $x$ and $y$. So is the inequality (\ref{sym2}) below. Also note that by the compactness of $X$,
$$
\mu(B(x,r)) \lesssim 1, \quad \forall x\in X, \forall r>0.
$$

In the proof of Theorem \ref{NIS}, we will use an equivalent definition for the singular integral operators, which involves the notion of a bounded set of elementary operators:

\begin{definition}[\cite{BookII}]
Let $\mathcal{E} = \{(E, 2^{-j})\} \subseteq C^\infty(X\times X) \times \{2^{-j}:j\in \mathbb{N}\}$ be a set of continuous operators $E: C^\infty(X)' \to C^\infty(X)$ each paired with a number $2^{-j}$. We say $\mathcal{E}$ is a bounded set of pre-elementary operators if for all ordered multi-indices $\alpha, \beta$, and for all $k\in \mathbb{N}$, 
\begin{equation}\label{sym2}
\big| W_x^\alpha W_y^\beta E(x,y) \big| \lesssim_{\alpha, \beta, k} 2^{j\,\text{deg}\,\alpha} \,2^{j\,\text{deg}\,\beta}\, \frac{(1+2^j\rho(x,y))^{-k}}{\mu(B(x, 2^{-j} + \rho(x,y)))}, \quad \forall (E, 2^{-j}) \in \mathcal{E}.    
\end{equation}
\end{definition}

\begin{definition}[\cite{BookII}]\label{elementary}
We define the set of bounded sets of elementary operators $\mathcal{G}$ to be the largest set of subsets of $C^\infty(X\times X) \times \{2^{-j}: j\in \mathbb{N}\}$ such that for all $\mathcal{E} \in \mathcal{G}$,
\begin{itemize}
    \item 
    $\mathcal{E}$ is a bounded set of pre-elementary operators.
    
    \item\label{0-elementary}
    There exist finite sets depending on $\mathcal{E}$:
    \begin{align*}
    & \mathcal{F}_1 = \big\{\big(b_{1,1}(x) W^{\alpha_{1,1}}, \kappa_{1,1}\big), \big(b_{1,2}(x) W^{\alpha_{1,2}}, \kappa_{1,2}\big),  \ldots, \big(b_{1, L_1}(x)W^{\alpha_{1, L_1}}, \kappa_{1,L_1}\big)\big\},\\
    & \mathcal{F}_2 = \big\{\big(b_{2,1}(x) W^{\alpha_{2,1}}, \kappa_{2,1}\big), \big(b_{2,2}(x) W^{\alpha_{2,2}}, \kappa_{2,2}\big), \ldots, \big(b_{2, L_2}(x)W^{\alpha_{2, L_2}}, \kappa_{2,L_2}\big)\big\},
    \end{align*}
    where each $b_{a, l}\in C^\infty(X), \text{deg}\,\alpha_{a,l} \leq \kappa_{a,l} \in \mathbb{N}$, and such that $\forall (E, 2^{-j})\in \mathcal{E}$,
    $$
    E= \sum_{\substack{1\leq l_1 \leq L_1 \\ 1\leq l_2 \leq L_2}} 2^{-(\epsilon_{1,l_1} + \epsilon_{2, l_2} +\kappa_{1, l_1}+ \kappa_{2, l_2})j}\, \big(b_{1,l_1}W^{\alpha_{1,l_1}}\big) \, E_{l_1, l_2}\, \big(b_{2, l_2} W^{\alpha_{2,l_2}}\big),
    $$
    where
    $$
    \epsilon_{a,l} =
    \left\{
    \begin{aligned}
    &1, \quad \kappa_{a,l}=0,\\
    &0, \quad \text{otherwise},
    \end{aligned}
    \right.
    $$
    and $\{(E_{l_1, l_2}, 2^{-j}): (E, 2^{-j})\in \mathcal{E}, 1\leq l_1\leq L_1, 1\leq l_2 \leq L_2\}\in \mathcal{G}$.
    \end{itemize}
\end{definition}

See \cite{BookII} for existence and uniqueness of $\mathcal{G}$. If $\mathcal{G}_0$ satisfies the above two conditions, then one must have $\mathcal{G}_0 \subseteq \mathcal{G}$. Any subset of a bounded set of elementary operators is still a bounded set of elementary operators. Any finite union of bounded sets of elementary operators is still a bounded set of elementary operators.

\begin{proposition}[\cite{BookII}]\label{sum}
For $t\in \mathbb{R}$, a continuous operator $T: C^\infty(X) \to C^\infty(X)'$ is a singular integral operator of order $t$ if and only if there exists a bounded set of elementary operators $\{(E_j, 2^{-j}):j\in \mathbb{N}\}$ with
$$
T = \sum_{j\in \mathbb{N}} 2^{jt} E_j.
$$
\end{proposition}

We also use the notion of bounded sets of elementary operators to define non-isotropic Besov and Triebel-Lizorkin spaces adapted to the Carnot-Carath\'eodory geometry:
\begin{definition}[\cite{BookII}]\label{triebel}
For $t\in \mathbb{R}, p\in (1, \infty), q\in (1, \infty]$, we define the non-isotropic Triebel-Lizorkin space on $(X,\mu)$ adapted to the Carnot-Carath\'eodory geometry: $\mathcal{F}_{p,q}^t(X, (W,d))$ to be the space of all $f\in C^\infty(X)'$ such that for every bounded set of elementary operators $\mathcal{E}$,
$$
\sup_{\{(E_j, 2^{-j}):j\in \mathbb{N}\}\subseteq \mathcal{E}}\big\| \big\{ 2^{jt} E_j f\big\}_{j\in \mathbb{N}}\big\|_{L^p((X,\mu), l^q(\mathbb{N}))} <\infty,
$$
equipped with the family of semi-norms 
$$
\Big\{\sup_{\{(E_j, 2^{-j}):j\in \mathbb{N}\}\subseteq \mathcal{E}}\big\| \big\{ 2^{jt} E_j f\big\}_{j\in \mathbb{N}}\big\|_{L^p(l^q)} : \mathcal{E} \text{ bounded set of elementary operators}\Big\}.
$$
For $q=2$, denote $\mathcal{F}_{p,2}^t(X,(W,d))$ by $NL^p_t(X, (W,d))$, and we call it the non-isotropic $L^p$ Sobolev space adapted to the Carnot-Carath\'eodory geometry.
\end{definition}

\begin{definition}[\cite{BookII}]\label{2.7}
For $t\in \mathbb{R}, p,q\in [1, \infty]$, we define the non-isotropic Besov space on $(X,\mu)$ adapted to the Carnot-Carath\'eodory geometry: $\mathcal{B}_{p,q}^t(X, (W,d))$ to be the space of all $f\in C^\infty(X)'$ such that for every bounded set of elementary operators $\mathcal{E}$,
$$
\sup_{\{(E_j, 2^{-j}):j\in \mathbb{N}\}\subseteq \mathcal{E}}\big\| \big\{ 2^{jt} E_j f\big\}_{j\in \mathbb{N}}\big\|_{l^q(\mathbb{N}, L^p(X,\mu))} <\infty,
$$
equipped with the family of semi-norms 
$$
\Big\{\sup_{\{(E_j, 2^{-j}):j\in \mathbb{N}\}\subseteq \mathcal{E}}\big\| \big\{ 2^{jt} E_j f\big\}_{j\in \mathbb{N}}\big\|_{l^q(L^p)} : \mathcal{E} \text{ bounded set of elementary operators}\Big\}.
$$
\end{definition}

To give the other standard definition for the non-isotropic Besov and Triebel-Lizorkin spaces, we need the following lemma.
\begin{lemma}\label{compact2}
The spectrum of $L$ is contained in $[0,\infty)$ and is discrete.
\end{lemma}

\begin{proof}
Since functions $\frac{1}{1+\lambda}, \frac{\lambda}{1+\lambda}$ are bounded on $[0, \infty)$, the operators $(I+L)^{-1}, L \circ (I+L)^{-1}$ are bounded operators on $L^2(X)$. By the maximal subellipticity,
\begin{align*}
\sum_{i=1}^\nu \big\|W_i^{h_i} (I+L)^{-1} f\big\|_{L^2(X)}^2 \lesssim \big\langle L(I+L)^{-1} f, (I+L)^{-1} f \big\rangle + \big\|(I+L)^{-1} f\big\|_{L^2(X)}^2  \leq 2\|f\|_{L^2(X)}^2.
\end{align*}
By Chapter 8 of Street \cite{BookII}, $\exists \epsilon>0$,
$$
\big\|(I+L)^{-1} f\big\|_{L_\epsilon^2(X)}^2 \lesssim \|f\|_{L^2(X)}^2.
$$
Hence $(I+L)^{-1}$ is a compact operator on $L^2(X)$. And since $(I+L)^{-1}$ is self-adjoint, the spectrum of $(I+L)^{-1}$ is countable and can accumulate only at $0$. Thus $L$ has discrete spectrum.
\end{proof}

Let $\lambda_0$ be the smallest nonzero spectrum. Fix $\phi_0 \in C_c^\infty[-\lambda_0/2, \lambda_0/2]$ with $0\leq \phi_0 \leq 1$, such that $\phi_0 \equiv 1$ on $[0, \lambda_0/4]$. Let 
\begin{equation}\label{phi}
\phi= (\phi_0- \phi_0(2^{2\kappa}\cdot)) \chi_{[0,\infty)}, \quad \phi_j = \phi \text{ for } j>0.
\end{equation}
Then $\{\phi_j: j>0\}$ is bounded in $C_c^\infty[2^{-2\kappa -2}\lambda_0, \lambda_0/2]$, and 
\begin{equation}\label{partition}
\sum_{j\in \mathbb{N}} \phi_j\big(2^{-2\kappa j}\cdot\big) \equiv 1, \quad \text{ on } [0,\infty).
\end{equation}
By Lemma \ref{etaiselementary}, each $\phi_j(2^{-2\kappa j} L)$ is continuous from $C^\infty(X)' \to C^\infty(X)$.

\begin{definition}\label{other}
For $t\in \mathbb{R}, p\in (1, \infty), q\in (1, \infty]$, we define the non-isotropic Triebel-Lizorkin space on $(X,\mu)$ adapted to $L$: $\mathcal{F}_{p,q}^t(X, L)$ to be the space of all $f\in C^\infty(X)'$ with the following norm finite:
$$
\Big\| \Big\{ 2^{jt} \phi_j(2^{-2\kappa j} L) f \Big\}_{j\in \mathbb{N}} \Big\|_{L^p(X,l^q(\mathbb{N}))}.
$$
When $q=2$, $\mathcal{F}_{p,2}^t(X,L)$ coincides with the space $NL^p_t(X, L)$ in Definition \ref{simpleversion}, and we call it the non-isotropic $L^p$ Sobolev space adapted to $L$.
\end{definition}

\begin{definition}
For $t\in \mathbb{R}, p, q\in [1, \infty]$, we define the non-isotropic Besov space on $(X,\mu)$ adapted to $L$: $\mathcal{B}_{p,q}^t(X, L)$ to be the space of all $f\in C^\infty(X)'$ with the following norm finite:
$$
\Big\| \Big\{ 2^{jt} \phi_j(2^{-2\kappa j} L) f \Big\}_{j\in \mathbb{N}} \Big\|_{l^q(\mathbb{N}, L^p(X))}.
$$
\end{definition}

\begin{theorem}\label{full-equivalence}
For $t\in \mathbb{R}, p\in (1,\infty), q\in (1, \infty]$,
$$
\mathcal{F}_{p,q}^t(X, (W, d)) = \mathcal{F}_{p,q}^t(X, L),
$$
with the same topology, and for $t\in \mathbb{R}, p,q\in [1,\infty]$,
$$
\mathcal{B}_{p,q}^t(X, (W, d)) = \mathcal{B}_{p,q}^t(X, L),
$$
with the same topology.
\end{theorem}

Theorem \ref{sobolev-equivalence} follows from Theorem \ref{full-equivalence}.

By the above theorem, since $\mathcal{F}^t_{p,q}(X,(W,d))$ and $\mathcal{B}^t_{p,q}(X,(W,d))$ in Definitions \ref{triebel} and \ref{2.7} do not involve $L$, we have

\begin{corollary}\label{indep}
Given $X$ and $(W,d)$, the spaces $\mathcal{F}_{p,q}^t(X,L)$ and $\mathcal{B}_{p,q}^t(X,L)$ do not depend on the choice of the non-negative, self-adjoint operator $L$ which is maximally subelliptic with respect to $(W,d)$.
\end{corollary}

Theorem \ref{unconditional} will follow from the following theorem via an interpolation result adapted from Calder\'on-Torchinsky multiplier theorem (see Chapter 5 of Grafakos \cite{grafakosbook}).
\begin{theorem}\label{Sobolev-bound}
Let $t, \tau\in \mathbb{R}, 1<p<\infty$. Denote $\Lambda^t = \max\{\lambda^t, (1+\lambda)^t\}$, as a function of $\lambda$ on $(0,\infty)$. $m(L)$ is bounded from $NL^p_{2\kappa \tau}(X)$ to $NL^p_{2\kappa \tau-2\kappa t}(X)$, if 
$$
\Big\| \frac{m(\lambda)}{\Lambda^t} \Big\|_{L^\infty_{\frac{Q}{2}+\epsilon, \text{sloc}}}<\infty, \quad \text{for some } \epsilon>0.
$$
Moreover, 
$$
\big\|m(L) \big\|_{NL^p_{2\kappa \tau}(X) \to NL^p_{2\kappa \tau-2\kappa t}(X)} \lesssim_{t, \tau, p, \epsilon} |m(0)| + \Big\| \frac{m(\lambda)}{\Lambda^t} \Big\|_{L^\infty_{\frac{Q}{2}+\epsilon, \text{sloc}}}.
$$
\end{theorem}

\section{A bounded set of elementary operators}\label{keylemma}
In this section, we study a bounded set of elementary operators as in the following lemma, which plays a central role in proving the equivalence between spaces (Theorem \ref{full-equivalence}) and in proving that $m(L)$ is a singular integral operator when $m$ satisfies an appropriate Mihlin-H\"ormander type condition (Theorem \ref{NIS}).

\begin{lemma}\label{etaiselementary}
Suppose $\eta_0$ is a bounded Borel measurable function supported on $[-\lambda_0/2, \lambda_0/2]$, and $\{\eta_j\}_{j>0}$ is bounded in $C_c^\infty[2^{-2\kappa -2}\lambda_0, \lambda_0/2]$. Then $\big\{\big(\eta_j(2^{-2\kappa j} L), 2^{-j}\big) \big\}_{j\in \mathbb{N}}$ is a bounded set of elementary operators. In particular, each $\eta_j(2^{-2\kappa j} L): C^\infty(X)' \to C^\infty(X)$ is continuous.
\end{lemma}

To prove Lemma \ref{etaiselementary}, we need estimates for kernels of various spectral multipliers. By Chapter 8 of Street \cite{BookII}, $L_0$ being maximal subelliptic (see (\ref{maxsub})), is equivalent to that a short time heat kernel Gaussian bound holds for any (equivalently: for some) non-negative self-adjoint extension $L$: $\forall 0<t \leq 1$, $e^{-tL}(x,y)\in C^\infty(X\times X)$, and $\exists c>0, \forall$ ordered multi-indices $\alpha, \beta$,
\begin{equation*}
\big| W_x^\alpha W_y^\beta  e^{-tL}(x,y) \big| \lesssim_{\alpha, \beta} \big(\rho(x,y) + t^{1/2\kappa} \big)^{-\text{deg}\,\alpha-\text{deg}\,\beta} e^{-c \big( \rho(x,y)^{2\kappa}/t \big)^{1/(2\kappa -1)}} \mu\big(B(x,\rho(x,y) + t^{1/2\kappa})\big)^{-1}.    
\end{equation*}
This Gaussian bound implies the following lemmas, via the same argument as in the proof of Lemmas 2.1, 2.2, 4.1, 4.3 in \cite{DUONG}. 

\begin{Lemma}[\cite{DUONG}]\label{plancherel-estimate}
For all $t\in (0, 1]$, and for all ordered multi-indices $\alpha$,
$$
\int \big|W_x^\alpha e^{-tL}(x,y)\big|^2\,d\mu(y) \lesssim_\alpha \mu(B(x, t^{1/2\kappa}))^{-1} \,t^{-\text{deg}\,\alpha/\kappa}.
$$
For all bounded Borel measurable function $\eta$ supported in $[-\lambda_0/2, \lambda_0/2]$, for all $j\in \mathbb{N}$, and for all ordered multi-indices $\alpha$,
$$
\int \big|W_x^\alpha \eta(2^{-2\kappa j} L) (x,y) \big|^2\,d\mu(y) \lesssim_\alpha \|\eta\|_{L^\infty(\mathbb{R})}^2 \,\mu(B(x, 2^{-j}))^{-1} \,2^{2j\,\text{deg}\,\alpha}.
$$
\end{Lemma}

\begin{Lemma}[\cite{DUONG}]\label{plancherel-decay}
For all $j, a \in \mathbb{N}$, and for all ordered multi-indices $\alpha$,
$$
\int \big|W_x^\alpha e^{ - 2^{-2\kappa j}L} (x,y) \big|^2 (1+ 2^j \rho(x,y))^a \,d\mu(y) \lesssim_{\alpha, a} \mu(B(x, 2^{-j}))^{-1}  \,2^{2j\,\text{deg}\,\alpha}.
$$
For all bounded Borel measurable function $\eta$ supported in $[-\lambda_0/2, \lambda_0/2]$, for all $j, a \in \mathbb{N}$, for all ordered multi-indices $\alpha$, and for all $\epsilon>0$,
$$
\int \big|W_x^\alpha \eta(2^{-2\kappa j}L) (x,y) \big|^2(1+2^j\rho(x,y))^{a}\,d\mu(y) \lesssim_{\alpha,a, \epsilon} \|\eta\|_{L^\infty_{\frac{a}{2}+\epsilon}(\mathbb{R})}^2 \,\mu(B(x, 2^{-j}))^{-1} \,2^{2j\,\text{deg}\,\alpha}.
$$
\end{Lemma}

Note for every bounded Borel measurable function $\eta$ on $\mathbb{R}$,
$$
\bar \eta(L)(x,y) = \overline{\eta(L)(y,x)},
$$
where $\bar \eta$ denotes the complex conjugate of $\eta$. Thus when replacing $W_x$ with $W_y$ and replacing $d\mu(y)$ with $d\mu(x)$ in the above two lemmas, the corresponding estimates also hold.

Denote by $E_{L}$ the spectral resolution of identity for $L$. We also need estimates on the kernel of $E_L(0)$:  
\begin{Lemma}\label{one-pre-elementary}
$\{(E_{L}(0),1)\}$ is a bounded set of elementary operators.
\end{Lemma}

\begin{proof}
Note $E_{L}(0)$ is an orthogonal projection on $L^2(X)$, and $\text{Range}\,(E_{L}(0)) = \text{Null}\,(I-E_{L}(0)) =\text{Null}\,(L)$. Since $I-E_L(0)$ is bounded on $L^2(X)$, $\text{Range}\,(E_{L}(0))$ is closed. If $Lu =0$, by the hypoellipticity of $L$, $u\in C^\infty(X)$. Thus $\text{Range}\,(E_{L}(0))\subseteq C^\infty(X)$. Since $(I+L)^{-1}$ is compact, $\text{Range}\,((I+L)^{-1})$ contains no infinitely-dimensional closed subspace. Hence $\text{Range}\,(E_{L}(0)) \subseteq \text{Range}\,((I+L)^{-1})$ is finite-dimensional. Let $\{e_1, \ldots, e_q\}$ be an orthonormal basis of $\text{Range}\,(E_{L}(0)) \subseteq C^\infty(X)\subseteq L^2(X)$. We have
\begin{align*}
E_{L}(0)f(x)= \sum_{l=1}^q \langle f, e_l \rangle e_l(x) = \int \sum_{l=1}^q e_l(x) \overline{e_l(y)} f(y)\,d\mu(y).
\end{align*}
Thus $E_{L}(0)(x,y) = \sum_{l=1}^q e_l(x) \overline{e_l(y)}$.
Since $X$ is compact, for all ordered multi-indices $\alpha$, $\beta$, and for all $k\in \mathbb{N}$,
$$
\Big| W_x^\alpha W_y^\beta E_{L}(0)(x,y) \Big| \lesssim_{\alpha,\beta} 1\lesssim_k (1+\rho(x,y))^{-k} \mu\big(B(x,1+\rho(x,y))\big)^{-1}.
$$
And since $E_L(0) = E_L(0)$, by Definition \ref{elementary}, $\{(E_{L}(0),1)\}$ is a bounded set of elementary operators.
\end{proof}

\begin{proof}[Proof of Lemma \ref{etaiselementary}]

Let $\mathcal{G}_0$ consist of all the sets $\mathcal{E} = \big\{\big(\eta_\gamma(2^{-2\kappa j_\gamma} L), 2^{-j_\gamma}\big): \gamma \in \Gamma\big\}$ satisfying that every $j_\gamma \in \mathbb{N}$, every $\eta_{\gamma}$ with $j_\gamma =0$ is a bounded Borel measurable function supported on $[-\lambda_0/2, \lambda_0/2]$, $\{\eta_\gamma(0): j_\gamma =0\}$ is bounded, and that $\{\eta_\gamma: j_\gamma>0\}$ is bounded in $C_c^\infty[2^{-2\kappa -2}\lambda_0, \lambda_0/2]$. Note $\{(\eta_j(2^{-2\kappa j}L), 2^{-j}):j\in \mathbb{N}\}$ given in Lemma \ref{etaiselementary} belongs to $\mathcal{G}_0$.

Fix an arbitrary $\mathcal{E} = \big\{\big(\eta_\gamma(2^{-2\kappa j_\gamma} L), 2^{-j_\gamma}\big): \gamma \in \Gamma\big\} \in \mathcal{G}_0$. We first show $\mathcal{E}$ is a bounded set of pre-elementary operators. I.e., we need to show that $\forall$ordered multi-indices $\alpha, \beta$, $\forall k\in \mathbb{N}$, 
$$
    \big| W_x^\alpha W_y^\beta \eta_\gamma(2^{-2\kappa j_\gamma} L) (x,y) \big| \lesssim_{\alpha, \beta, k} 2^{j_\gamma\,\text{deg}\,\alpha} 2^{j_\gamma \,\text{deg}\,\beta} \frac{(1+2^{j_\gamma}\rho(x,y))^{-k}}{\mu(B(x, 2^{-j_\gamma} + \rho(x,y)))}, \quad \forall \gamma \in \Gamma.
$$
For $\gamma \in \Gamma$ with $j_\gamma =0$, $\eta_\gamma(L) = \eta_\gamma (0)E_{L}(0)$. Since $\eta_\gamma(0)$ is bounded uniformly for $\gamma$ with $j_\gamma =0$, by Lemma \ref{one-pre-elementary}, $\{(\eta_\gamma(L), 1): \gamma\in \Gamma, j_\gamma =0\}$ is a bounded set of pre-elementary operators. As for $\gamma\in \Gamma$ with $j_\gamma >0$, we divide into two cases:

\textbf{Case 1. on-diagonal:} $ \rho(x,y) \leq 2^{-j_\gamma}$.

By Lemma \ref{plancherel-estimate} and the doubling condition (\ref{doub}), for every $\gamma$ with $j_\gamma>0$,
\begin{align*}
&\quad \big| W_x^\alpha  W_y^\beta \eta_\gamma(2^{-2\kappa j_\gamma} L)(x,y)\big| & \\
& \leq \Big(\int \Big|W_x^\alpha \sqrt{|\eta_\gamma|}(2^{-2\kappa j_\gamma} L)(x,z)\Big|^2\, dz \Big)^{1/2} \cdot \Big( \int \Big| W_y^\beta \big(\sqrt{|\eta_\gamma|} \text{ sign}\,\eta_\gamma\big)(2^{-2\kappa j_\gamma} L)(z,y) \Big|^2\,dz \Big)^{1/2}\\
& \lesssim_{\alpha, \beta} 2^{j_\gamma \,\text{deg}\,\alpha + j_\gamma\,\text{deg}\,\beta} \big\| \sqrt{|\eta_\gamma|} \big\|_{L^\infty(\mathbb{R})}^2 \mu(B(x,2^{-j_\gamma}))^{-1/2} \mu(B(y, 2^{-j_\gamma}))^{-1/2}\\
& \lesssim 2^{j_\gamma \,\text{deg}\,\alpha + j_\gamma\,\text{deg}\,\beta} \big\| \eta_\gamma \big\|_{L^\infty(\mathbb{R})} \mu(B(x,2^{-j_\gamma}))^{-1} (1+ 2^{j_\gamma}\rho(x,y))^{Q/2}\\
& \lesssim_k 2^{j_\gamma \,\text{deg}\,\alpha + j_\gamma\,\text{deg}\,\beta}  \big\| \eta_\gamma \big\|_{L^\infty(\mathbb{R})} \mu(B(x,2^{-j_\gamma}+\rho(x,y)))^{-1} (1+ 2^{j_\gamma}\rho(x,y))^{-k}.
\end{align*}

\textbf{Case 2. off-diagonal:} $\rho(x,y) \geq 2^{-j_\gamma}$.

Let $\psi_\gamma(\lambda) = \eta_\gamma(\lambda) e^\lambda$. Then $\eta_\gamma(\lambda) = \psi_\gamma(\lambda) e^{-\lambda}$, and $\{\psi_\gamma: j_\gamma>0\}$ is bounded in $C_c^\infty([-\lambda_0/2, \lambda_0/2])$.

%Fix $x,y$ now. By Lemma \ref{plancherel-decay} and doubling condition, we have
%\begin{equation}
%\begin{aligned}
%&\quad \big( \int_{\rho(z,y) \geq \frac{1}{3} \rho(x,y)} |W_y^\beta \psi_\gamma(2^{-2\kappa j} L)(z,y)|^2 \, dz \big)^{1/2}\\
%&\lesssim \int |\hat G(\tau)| \big( \int_{\rho(z,y) \geq \frac{1}{3} \rho(x,y)} |W_y^\beta e^{(i\tau-1) 2^{-2\kappa j} L} (z,y)|^2\, dz \big)^{1/2}\, d\tau\\
%&\lesssim_{\beta, a} \int |\hat G(\tau)| (1+ 2^j \rho(x,y))^{-a/2} \mu(y, 2^{-j})^{-1/2} 2^{|\beta | j} (1+ |\tau|)^{a/2} \,d\tau\\
%&\lesssim (1+ 2^j \rho(x,y))^{-a/2} \mu(y, 2^{-j})^{-1/2} 2^{|\beta | j} \big(\int |\hat G(\tau)|^2  (1+ |\tau|^2)^{a/2+1} \,d\tau \big)^{1/2}\\
%&\lesssim (1+ 2^j \rho(x,y))^{-a/2} \mu(y, 2^{-j})^{-1/2} 2^{|\beta | j} \,(\sup_j \|\psi_\gamma\|_{C^{a/2+1}})\\
%&\lesssim_a (1+ 2^j \rho(x,y))^{-a/2 + Q/2} \mu(y, \rho(x,y))^{-1/2} 2^{|\beta | j} .
%\end{aligned}
%\end{equation}
%Similarly, we have
%$$
%\big( \int_{\rho(x,z) \geq \frac{1}{3} \rho(x,y)} |W_x^\alpha \psi_\gamma(2^{-2\kappa j} L)(x,z)|^2 \, dz \big)^{1/2} \lesssim_{\beta, a} (1+ 2^j \rho(x,y))^{-a/2 + Q/2} \mu(x, \rho(x,y))^{-1/2} 2^{|\alpha | j}
%$$
%Similarly, 
%$$
%\big( \int_{\rho(z,x) \geq \frac{1}{3} \rho(x,y)} |W_x^\alpha \psi_\gamma(2^{-2\kappa j} L)(x,z)|^2 \, dz \big)^{1/2} \lesssim_{\alpha, a} (1+ 2^j \rho(x,y))^{-a/2 + Q/2} \mu(x, \rho(x,y))^{-1/2} 2^{|\alpha | j} .
%$$

By Lemmas \ref{plancherel-estimate} and \ref{plancherel-decay}, and the doubling condition (\ref{doub}), we have for $\gamma$ with $j_\gamma>0$,
\begin{align*}
&\quad \big|  W_x^\alpha  W_y^\beta \eta_\gamma(2^{-2\kappa j_\gamma}L) (x,y) \big|\\
& = \Big|\int W_x^\alpha \psi_\gamma(2^{-2\kappa j_\gamma}L)(x,z) W_y^\beta e^{-2^{-2\kappa j_\gamma}L} (z,y)\,dz \Big|\\
& \leq \int_{\rho(y,z) \geq \frac{1}{3}\rho(x,y)} \big| W_x^\alpha \psi_\gamma(2^{-2\kappa j_\gamma}L)(x,z) W_y^\beta e^{-2^{-2\kappa j_\gamma}L} (z,y)\big|\,dz  \\
& \quad + \int_{\rho(x,z)\geq \frac{1}{3}\rho(x,y)} \big| W_x^\alpha \psi_\gamma(2^{-2\kappa j_\gamma}L)(x,z) W_y^\beta e^{-2^{-2\kappa j_\gamma}L} (z,y)\big|\,dz \\
& \leq \Big(\int \big| W_x^\alpha \psi_\gamma(2^{-2\kappa j_\gamma}L)(x,z) \big|^2\,dz\Big)^{1/2} \Big( \int_{\rho(y,z) \geq \frac{1}{3}\rho(x,y)} \big|W_y^\beta e^{-2^{-2\kappa j_\gamma}L} (z,y)\big|^2\,dz \Big)^{1/2} \\
& \quad + \Big(\int_{\rho(x,z)\geq \frac{1}{3}\rho(x,y)} \big|W_x^\alpha \psi_\gamma(2^{-2\kappa j_\gamma}L)(x,z)\big|^2 \,dz\Big)^{1/2} \Big(\int \big|W_y^\beta e^{-2^{-2\kappa j_\gamma}L} (z,y)\big|^2\,dz \Big)^{1/2}\\
&\lesssim_{\alpha, \beta, k, \epsilon} 2^{j_\gamma\,\text{deg}\,\alpha + j_\gamma\,\text{deg}\,\beta} \|\psi_\gamma\|_{L^\infty_{k+Q+\epsilon/2}(\mathbb{R})} (1+ 2^{j_\gamma} \rho(x,y))^{-k-Q } \mu(B(x, 2^{-j_\gamma}))^{-1/2} \mu(B(y, 2^{-j_\gamma}))^{-1/2}\\
&\lesssim_{k,\epsilon} 2^{j_\gamma\,\text{deg}\,\alpha + j_\gamma\,\text{deg}\,\beta} \|\eta_\gamma\|_{L^\infty_{k+Q+\epsilon}(\mathbb{R})} (1+ 2^{j_\gamma} \rho(x,y))^{-k} \mu(B(x, 2^{-j_\gamma} + \rho(x,y)))^{-1}.
\end{align*}
Therefore $\mathcal{E}$ is a bounded set of pre-elementary operators. In particular, each $\eta_\gamma(2^{-2\kappa j_\gamma} L) \in C^\infty(X\times X)$.

Next we verify the second condition in Definition \ref{elementary} for $\mathcal{G}_0$. For $\gamma$ with $j_\gamma>0$, let $\widetilde \eta_\gamma(\lambda) := \frac{\eta_\gamma(\lambda)}{\lambda^2}$. $\widetilde \eta_\gamma$ with $j_\gamma>0$ is still bounded in $C_c^\infty[2^{2\kappa -2} \lambda_0, \lambda_0/2]$.
(\ref{form}) implies that when acting on $C^\infty(X)$,
$$
L = \sum_{\text{deg}\,\alpha \leq 2\kappa} b_\alpha(x) W^\alpha, \quad \text{for some } b_\alpha \in C^\infty(X).
$$
Therefore 
\begin{align*}
\eta_\gamma(2^{-2\kappa j_\gamma}L) &= (2^{-2\kappa j_\gamma}L) \, \widetilde \eta_\gamma(2^{-2\kappa j_\gamma}L) \, (2^{-2\kappa j_\gamma}L) \\
&=\Big(2^{-2\kappa j_\gamma} \sum_{\text{deg}\,\alpha\leq 2\kappa} b_\alpha(x) W^\alpha \Big) \widetilde \eta_\gamma( 2^{-2\kappa j_\gamma} L) \Big(2^{-2\kappa j_\gamma} \sum_{\text{deg}\,\beta\leq 2\kappa} b_\beta(x) W^\beta\Big)\\
&= \sum_{\text{deg}\,\alpha, \text{deg}\,\beta \leq 2\kappa} 2^{-4\kappa j_\gamma} \,\big(b_\alpha(x)  W^\alpha\big) \, \widetilde \eta_\gamma( 2^{-2\kappa j_\gamma} L)\, \big(b_\beta(x) W^\beta\big).
\end{align*}
For $\gamma$ with $j_\gamma=0$, $\eta_\gamma(L) = \eta_\gamma(L)$. And since 
$$
\{(\eta_\gamma(L), 1):j_\gamma =0\} \cup \{(\widetilde \eta_\gamma(2^{-2\kappa j_\gamma} L), 2^{-j_\gamma}): j_\gamma>0\} \in \mathcal{G}_0,
$$
by Definition \ref{elementary}, $\mathcal{G}_0 \subseteq \mathcal{G}$. Thus $\{(\eta_j(2^{-2\kappa j}L), 2^{-j}):j\in \mathbb{N}\}\in \mathcal{G}_0$ is a bounded set of elementary operators.
\end{proof}

\section{Singular integral operators}\label{sectionnis}
In this section, we prove that $m(L)$ is a singular integral operator when $m$ satisfies an appropriate Mihlin-H\"ormander type condition (Theorem \ref{NIS}). Recall $m: [0,\infty) \to \mathbb{C}$ is a bounded Borel measurable function, and that $\lambda_0$ is the smallest nonzero spectrum of $L$. 

\begin{Lemma}\label{decompose-spectrum}
If
$$
\sup_{\lambda >0} \lambda^{-t} \big|(\lambda \partial_\lambda)^k m(\lambda) \big|<\infty, \quad \forall k\in \mathbb{N},    
$$
we can write
$$
m(\lambda) = \sum_{j\in \mathbb{N}} 2^{2\kappa t j} m_j(2^{-2\kappa j} \lambda), \quad \text{for } \lambda \in [0,\infty),
$$
where $\{m_j\}_{j>0}$ is bounded in $C_c^\infty[2^{-2\kappa -2} \lambda_0, \lambda_0/2]$, and $m_0$ is a bounded Borel measurable function supported on $[0, \lambda_0/2]$. %and 
%$$
%\sup_{j>0} \|m_j\|_{C^k} \lesssim_{k,t} \sum_{l\leq k} \sup_{\lambda>0} \big| \lambda^{-t} (\lambda \partial_\lambda )^l m(\lambda) \big|, \quad \|m_0\|_{t,k} \lesssim_{t,k} \sum_{l\leq k} \sup_{\lambda>0} \big| \lambda^{-t} (\lambda \partial_\lambda )^l m(\lambda) \big|,
%$$
%and for $r\in [1, \infty], k\in \mathbb{N}$,
%$$
%\sup_{j>0} \|m_j\|_{L^r_k} \leq \|m\|_{L^r_{k, \text{sloc}}}.
%$$

%might not need the following.
%Conversely, if 
%$$
%m(\lambda) = \sum_{j\in \mathbb{N}} 2^{2\kappa t j} m_j(2^{-2\kappa j} \lambda),
%$$
%with $\{m_j\}_{j>0}$ bounded in $C_c^\infty([2^{-2\kappa -2} \lambda_0, \lambda_0/2])$, and $m_0$ is supported in $[0, \lambda_0/2]$, then
%$$
%\sup_{\lambda >0} |\lambda^{-t} (\lambda \partial_\lambda)^a m(\lambda) | \lesssim_{t,a} \sup_{j>0} \|m_j\|_{C^k} + \sup_{\lambda>0} \lambda^{-t} \big| (\lambda \partial_{\lambda})^a m_0(\lambda) \big|.
%$$
\end{Lemma}

Before we prove this lemma, we first derive Theorem \ref{NIS} from it.

\begin{proof}[Proof of Theorem \ref{NIS}]
By Lemmas \ref{etaiselementary} and \ref{decompose-spectrum}, we can write
$$
m(L) = \sum_{j\in \mathbb{N}} 2^{2\kappa t j} m_j(2^{-2\kappa j} L),
$$
where $\big\{ \big(m_j(2^{-2\kappa j} L), 2^{-j}\big): j\in \mathbb{N} \big\}$ is a bounded set of elementary operators. Then by Proposition \ref{sum}, $m(L)$ is a singular integral operator of order $2\kappa t$. 
\end{proof}

\begin{proof}[Proof of Lemma \ref{decompose-spectrum}]
Recall $\{\phi_j\}_{j\in \mathbb{N}}$ satisfies (\ref{partition}). Let 
\begin{equation}\label{mj}
m_j(2^{-2\kappa j} \lambda) := 2^{-2\kappa t j} m(\lambda) \phi_j(2^{-2\kappa j} \lambda).   
\end{equation}
Then we have
$$
m(\lambda ) = m(\lambda) \sum_{j\in \mathbb{N}}  \phi_j(2^{-2\kappa j}\lambda) = \sum_{j\in \mathbb{N}} 2^{2\kappa t j} m_j(2^{-2\kappa j} \lambda).
$$
For $j>0$,  
\begin{align*}
&\quad \sup_{\lambda>0} \big|\partial_\lambda^k m_j(\lambda)\big| = \sup_{\lambda\in [2^{-2\kappa -2} \lambda_0, \lambda_0/2]} 2^{-2\kappa t j}\Big|\partial_\lambda^k \Big( m(2^{2\kappa j}\lambda) \phi_j(\lambda)\Big) \Big| \\
&\lesssim_k 2^{-2\kappa t j} \sum_{l\leq k} \sup_{\lambda\in [2^{-2\kappa -2} \lambda_0, \lambda_0/2]} \big|\partial_\lambda^l m(2^{2\kappa j}\lambda)\big| = 2^{-2\kappa t j} \sum_{l\leq k} \sup_{\lambda\in [2^{-2\kappa -2} \lambda_0, \lambda_0/2]} \Big|2^{2l\kappa j} (\partial^l m)(2^{2\kappa j}\lambda)\Big| \\
&\approx_{k,t}   \sum_{l\leq k} \sup_{\lambda\in [2^{-2\kappa -2} \lambda_0, \lambda_0/2]} (2^{2\kappa j}\lambda)^{-t} \Big|(2^{2\kappa j}\lambda)^l (\partial^l m)(2^{2\kappa j}\lambda)\Big|\\
&\leq \sum_{l\leq k} \sup_{\lambda>0} \lambda^{-t} \big|\lambda^l \partial_\lambda^l m(\lambda)\big| \approx_k \sum_{l\leq k} \sup_{\lambda>0} \lambda^{-t} \big|(\lambda \partial_\lambda)^l m(\lambda)\big|.
\end{align*}
\end{proof}

\section{Equivalence between spaces}\label{spaceequiv}
In this section, we will prove that the two standard ways to define non-isotropic Besov and Triebel-Lizorkin spaces are equivalent (Theorem \ref{full-equivalence}), from which Theorem \ref{sobolev-equivalence} follows. The proof is simple using the framework of Street \cite{BookII}. We first need the following ``Calder\'on Reproducing Formula''.
\begin{Lemma}\label{calderon}
For every $j\in \mathbb{N}$,
$$
\phi_j(2^{-2\kappa j} L) = \phi_j(2^{-2\kappa j} L) \sum_{\substack{l\in \mathbb{N}\\ |l-j|\leq 1}} \phi_l(2^{-2\kappa l} L).
$$
\end{Lemma}

\begin{proof}
By (\ref{partition}) and the way we define the $\phi_j$,
$$
\phi_j(2^{-2\kappa j} \lambda) = \phi_j(2^{-2\kappa j} \lambda) \sum_{\substack{l\in \mathbb{N}\\ |l-j|\leq 1}} \phi_l(2^{-2\kappa l} \lambda), \quad \text{for } \lambda \in [0,\infty).
$$
The conclusion then follows by the spectral theorem.
\end{proof}

We also need the following two lemmas.
\begin{lemma}[\cite{BookII}]\label{gain}
Let $\mathcal{E}$ be a bounded set of elementary operators, and pick
$$
\mathcal{F}:=\{(F_j, 2^{-j}): j\in \mathbb{N}\}, ~\mathcal{F}':=\{(F_j', 2^{-j}): j\in \mathbb{N}\} \subseteq \mathcal{E}.
$$
For every $k\in \mathbb{Z}$, define a vector-valued operator $\mathcal{T}_k$ by
$$
\mathcal{T}_k \{f_j\}_{j\in \mathbb{N}} := \{F_j F_{j+k}' f_j\}_{j\in \mathbb{N}}.
$$
Then for $p\in (1, \infty), q\in (1,\infty], N\in \mathbb{N}$,
$$
\sup_{\mathcal{F}, \mathcal{F}'} \|\mathcal{T}_k\|_{L^p(X, l^q(\mathbb{N})) \to L^p(X, l^q(\mathbb{N}))}\lesssim_{\mathcal{E}, p,q,N} 2^{-N|k|}, \quad \forall k\in \mathbb{Z}.
$$
and for $p,q \in [1,\infty], N\in \mathbb{N}$,
$$
\sup_{\mathcal{F}, \mathcal{F}'} \|\mathcal{T}_k\|_{l^q(\mathbb{N}, L^p(X)) \to l^q(\mathbb{N}, L^p(X))} \lesssim_{\mathcal{E}, p,q,N} 2^{-N|k|}, \quad \forall k\in \mathbb{Z}.
$$
\end{lemma}

\begin{lemma}[\cite{BookII}]\label{elementaryconverges}
If $\{(E_j, 2^{-j}): j\in \mathbb{N}\}$ is a bounded set of elementary operators, then $\forall g\in C^\infty(X)$, 
$$
\sum_{j\in \mathbb{N}} E_j \,g
$$
converges in $C^\infty(X)$.
\end{lemma}

\begin{proof}[Proof of Theorem \ref{full-equivalence}]
By Lemma \ref{etaiselementary}, $\{(\phi_j(2^{-2\kappa j} L), 2^{-j}): j\in \mathbb{N}\}$ is a bounded set of elementary operators. Denote 
$$
\mathcal{D} = \big\{(D_j, 2^{-j}):j\in \mathbb{N}\big\} := \big\{\big(\phi_j(2^{-2\kappa j}L), 2^{-j}\big): j\in \mathbb{N}\big\}.
$$
If $f\in NL_t^p(X,(W,d))$, then
$$
\big\|\big\{2^{tj} D_j f\big\}_{j\in \mathbb{N}} \big\|_{L^p(l^q(\mathbb{N}))} <\infty.
$$

On the other hand, let $\mathcal{E} = \{(E, 2^{-j})\}$ be an arbitrary bounded set of elementary operators. Denote $\mathcal{E}' = \mathcal{E} \cup \mathcal{D}$, which is still a bounded set of elementary operators. For every $j<0$, let $D_j =0$. 
By Lemma \ref{elementaryconverges} and by duality, for $f\in C^\infty(X)'$, each $D_jf \in C^\infty(X)$, and
$$
f= \sum_{j\in \mathbb{Z}} D_j f,
$$
where the infinite sum converges in $C^\infty(X)'$. By Lemmas \ref{calderon},
\begin{align*}
&\quad \sup_{\{(E_j, 2^{-j}):j\in \mathbb{N}\} \subseteq \mathcal{E}}\Big\| \big\{2^{jt} E_j f\big\}_{j\in\mathbb{N}}\Big\|_{L^p(l^q(\mathbb{N}))} \\
& = \sup_{\{(E_j, 2^{-j}):j\in \mathbb{N}\} \subseteq \mathcal{E}}\Big\| \Big\{2^{jt} E_j \sum_{k\in \mathbb{Z}} D_{j+k} f\Big\}_{j\in \mathbb{N}} \Big\|_{L^p(l^q(\mathbb{N}))} \\
&= \sup_{\{(E_j, 2^{-j}):j\in \mathbb{N}\} \subseteq \mathcal{E}}\Big\| \Big\{2^{jt} E_j \sum_{\substack{k,l\in \mathbb{Z}\\ |l|\leq 1}} D_{j+k} D_{j+k+l} f\Big\}_{j\in \mathbb{N}} \Big\|_{L^p(l^q(\mathbb{N}))} \\
& \leq \sup_{\{(E_j, 2^{-j}):j\in \mathbb{N}\} \subseteq \mathcal{E}}\sum_{\substack{k,l\in \mathbb{Z}\\ |l|\leq 1}} 2^{-t(k+l)} \Big\| \Big\{ E_j  D_{j+k} 2^{t(j+k+l)} D_{j+k+l} f\Big\}_{j\in \mathbb{N}} \Big\|_{L^p(l^q(\mathbb{N}))},
\end{align*}
where the last inequality follows from the triangle inequality and Fatou's lemma. Then by lemma \ref{gain},
\begin{align*}
\sup_{\{(E_j, 2^{-j}):j\in \mathbb{N}\} \subseteq \mathcal{E}}\Big\| \big\{2^{jt} E_j f\big\}_{j\in\mathbb{N}}\Big\|_{L^p(l^q(\mathbb{N}))}& \lesssim_{\mathcal{E}', t, p,q} \sum_{\substack{k,l\in \mathbb{Z}\\ |l|\leq 1}} 2^{-t(k+l)-(|t|+1)|k|} \Big\| \Big\{  2^{t(j+k+l)} D_{j+k+l} f\Big\}_{j\in \mathbb{N}} \Big\|_{L^p(l^q(\mathbb{N}))}\\
& \lesssim_t \Big\| \Big\{  2^{t(j+k+l)} D_{j+k+l} f\Big\}_{j\in \mathbb{Z}} \Big\|_{L^p(l^q(\mathbb{N}))} =  \Big\| \big\{2^{jt} D_j f\big\}_{j\in\mathbb{N}}\Big\|_{L^p(l^q(\mathbb{N}))}.
\end{align*}
Therefore for $t\in \mathbb{R}, p\in (1,\infty), q\in (1, \infty]$,
$$
\mathcal{F}_{p,q}^t(X, (W, d)) = \mathcal{F}_{p,q}^t(X, L).
$$

The proof for non-isotropic Besov spaces is similar.
\end{proof}

From now on, we denote $NL^p_t(X,(W,d))=NL^p_t(X,L)$ as $NL^p_t(X)$.

\section{Rough multipliers}\label{sectionbound}

In this section we quantify the amount of smoothness of $m$ required for the boundedness of the multiplier $m(L)$ on each $NL^p_t(X)$ space (Theorems \ref{unconditional} and \ref{Sobolev-bound}). We will combine the result of Duong, Ouhabaz, and Sikora \cite{DUONG} and the framework of Street \cite{BookII}, together with an interpolation result in Appendix \ref{B}.

\begin{proposition}[Duong, Ouhabaz, and Sikora \cite{DUONG}]\label{L^p-bound}
We have $m(L)$ is bounded on $L^p(X)$ for $1<p<\infty$ if $m\in L^\infty_{Q/2+\epsilon, \text{sloc}}$ for some $\epsilon>0$. Moreover,  
$$
\big\|m(L)\big\|_{L^p(X) \to L^p(X)} \lesssim_{p,\epsilon} |m(0)| + \|m\|_{L^\infty_{\frac{Q}{2}+\epsilon, \text{sloc}}}.
$$
\end{proposition}

\begin{proposition}[Street \cite{BookII}]\label{nisbound}
Let $t\in \mathbb{R}$. The singular integral operator of order $t$ is bounded from $NL^p_{\tau}(X) \to NL^p_{\tau-t}(X)$, for all $\tau\in \mathbb{R}$ and $1<p<\infty$, with norm only depending on $t, \tau, p$. 
\end{proposition}

\begin{lemma}\label{resolution}
Let $t,\tau\in \mathbb{R}$ and $1<p<\infty$. We have
$$
\big\|E_L(0) \big\|_{NL^p_{2\kappa \tau}(X) \to NL^p_{2\kappa \tau - 2\kappa t}(X)} \lesssim_{t, \tau,p} 1.
$$
\end{lemma}

\begin{proof}
By Lemma \ref{one-pre-elementary} and Proposition \ref{sum},
$$
E_L(0) = 2^{0\cdot t} E_L(0) + \sum_{j>0} 2^{jt} \cdot 0
$$
is a singular integral operator of order $t\in \mathbb{R}$. Then by Proposition \ref{nisbound}, the conclusion follows.
\end{proof}

\begin{proof}[Proof of Theorem \ref{Sobolev-bound}]
Recall $m_0(\lambda)$ is defined in (\ref{mj}). $m(\lambda)-m_0(\lambda)$ is supported on $[\lambda_0/4, \infty)$. Fix $\sigma\in C_c^\infty[\lambda_0/8, \infty)$, and $\sigma \equiv 1$ on $[\lambda_0/4, \infty)$. We have $\forall t\in \mathbb{R}, \forall k\in \mathbb{N}$, 
$$
\sup_{\lambda>0} \lambda^{-t} (\lambda \partial_\lambda)^k((1+\lambda)^t\sigma(\lambda))<\infty.
$$
By Theorem \ref{NIS}, $(I+L)^t\sigma(L)$ is a singular integral operator of order $2\kappa t$. By Propositions \ref{L^p-bound} and \ref{nisbound},
\begin{align*}
&\quad \|m(L)-m_0(L)\|_{NL^p_{2\kappa \tau}(X) \to NL^p_{2\kappa \tau- 2\kappa t}(X)} \\
& \leq \big\| (I+L)^{t-\tau} \sigma(L) \big\|_{L^p(X) \to NL^p_{2\kappa \tau - 2\kappa t}(X)} \cdot \big\|(I+L)^{-t}(m(L) - m_0(L))\big\|_{L^p(X) \to L^p(X)}  \\
&\qquad \qquad \qquad \qquad \qquad \qquad \qquad \qquad \qquad \cdot \big\|(I+L)^\tau \sigma(L)\big\|_{NL^p_{2\kappa \tau}(X) \to L^p(X)}\\
& \lesssim_{t,\tau,p} \big\|(I+L)^{-t}(m(L) - m_0(L))\big\|_{L^p(X) \to L^p(X)} \\
& \lesssim_{p,\epsilon,r} \big\|(1+\lambda)^{-t} (m(\lambda)-m_0(\lambda))\big\|_{L^r_{\frac{Q}{2}+\frac{\epsilon}{2}, \text{sloc}}}  \lesssim_{\epsilon} \Big\|\frac{ m(\lambda)}{\Lambda^t}\Big\|_{L^r_{\frac{Q}{2}+\epsilon, \text{sloc}}}.
\end{align*}

By Lemma \ref{resolution},
\begin{equation}\label{m0}
\big\|m_0(L) \big\|_{NL^p_{2\kappa \tau}(X) \to NL^p_{2\kappa \tau - 2\kappa t}(X)} = |m(0)| \big\|E_L(0) \big\|_{NL^p_{2\kappa \tau}(X) \to NL^p_{2\kappa \tau - 2\kappa t}(X)} \lesssim_{t, \tau,p} |m(0)|.
\end{equation}
Hence
$$
\big\|m(L) \big\|_{NL^p_{2\kappa \tau}(X) \to NL^p_{2\kappa \tau - 2\kappa t}(X)} \lesssim_{t, \tau, p,\epsilon,r} |m(0)| + \Big\|\frac{ m(\lambda)}{\Lambda^t}\Big\|_{L^r_{\frac{Q}{2}+\epsilon, \text{sloc}}}.
$$
\end{proof}

\begin{remark}
As mentioned in Remark \ref{1.8}, we cannot change $m$ on a set of measure $0$, and we use a continuous version of $m|_{(0,\infty)}$. The reason we cannot arbitrarily change the value of $m(0)$ is simply (\ref{m0}). As for $m(\lambda)$ for $\lambda \geq \lambda_0$, the reason we cannot change $m(\lambda)$ is included in the proof of Proposition \ref{L^p-bound} due to Duong, Ouhabaz, and Sikora (see also their key estimates in Lemma \ref{plancherel-decay}), where they used the following Fourier inversion formula to express $m_j$ defined in (\ref{mj}): 
$$
m_j(2^{-2\kappa j} \lambda) = \frac{1}{2\pi} \int \mathcal{F}\big(m_j(\lambda)e^\lambda\big)(\tau) e^{(i\tau-1)2^{-2\kappa j\lambda}}\,d\tau.
$$
The integrand on the right hand side is integrable, and thus this version of $m_j$ is continuous.
\end{remark}

We now derive Theorem \ref{unconditional} from Theorem \ref{Sobolev-bound} by an interpolation result: Proposition \ref{interpolation}.

\begin{proof}[Proof of Theorem \ref{unconditional}]
We first assume
$$
1<p<2, \quad Q\Big|\frac{1}{p}-\frac{1}{2}\Big| < s \leq \frac{Q}{2}, \quad 0 \leq \frac{1}{r} < \min \Big\{ 1, s-Q\Big|\frac{1}{p}-\frac{1}{2}\Big|\Big\}.
$$
Let 
\begin{align*}
&s_1= \frac{2}{3}\Big(s-Q\Big|\frac{1}{p}-\frac{1}{2}\Big|\Big) + \frac{1}{3}\cdot\frac{1}{r}>0, \quad s_0 = \frac{Q}{2} + \frac{1}{3}\Big(s-Q\Big|\frac{1}{p}-\frac{1}{2}\Big|\Big) + \frac{2}{3}\cdot\frac{1}{r}>\frac{Q}{2},\\
&\frac{1}{p_1} = \frac{1}{2}, \quad \frac{1}{p_1}<\frac{1}{p_0}=\frac{1}{p_1} + \frac{s_0-s_1}{Q} = 1- \frac{1}{3Q} \Big( s-Q\Big|\frac{1}{p}-\frac{1}{2}\Big| -\frac{1}{r}\Big)<1.    
\end{align*}
We have
\begin{align*}
Q\Big(\frac{1}{p} -\frac{1}{p_1}\Big) &= s- \Big( s-Q\Big| \frac{1}{p}-\frac{1}{2}\Big| \Big) < s-\Big[\frac{1}{3}\Big(s-Q\Big| \frac{1}{p} -\frac{1}{2}\Big|\Big) - \frac{1}{3}\cdot\frac{1}{r}- \Big(\frac{Q}{2}-s
\Big) \Big]   \\
&=s_0-s_1 =Q\Big(\frac{1}{p_0}-\frac{1}{p_1}\Big).
\end{align*}
Thus $\frac{1}{p_1}<\frac{1}{p}< \frac{1}{p_0}$, and $\exists 1<\theta<1$ such that 
$$
\frac{1}{p} = \frac{1-\theta}{p_0} +  \frac{\theta}{p_1}. 
$$
We have
$$
s- \big((1-\theta) s_0 + \theta s_1\big) =s-\Big(s_1+Q\Big(\frac{1}{p}-\frac{1}{p_1}\Big) \Big) = \frac{1}{3}\Big( s-Q\Big| \frac{1}{p}-\frac{1}{2}\Big| - \frac{1}{r}\Big)>0.
$$
Note $s_0, s_1>\frac{1}{r}$. By the Sobolev embedding theorem, since $p_1=2$,
$$
\big\|m(L)-m_0(L)\big\|_{L^{p_1}(X) \to L^{p_1}(X)} \leq \|m\|_{L^\infty(\mathbb{R})} \approx \|m\|_{L^\infty_{0, \text{sloc}}} \lesssim_{s_1, r} \|m\|_{L^{r}_{s_1, \text{sloc}}}.
$$
By Theorem \ref{Sobolev-bound} ($t, \tau=0$ case) and the Sobolev embedding theorem, by letting $\epsilon =\frac{s_0 - \frac{Q}{2}-\frac{1}{r}}{2}>0$, 
$$
\big\|m(L)-m_0(L)\big\|_{L^{p_0}(X) \to L^{p_0}(X)} \lesssim_{\epsilon,p_0} \|m\|_{L^\infty_{\frac{Q}{2}+\epsilon, \text{sloc}}} \lesssim_{r,s_0, \epsilon} \|m\|_{L^{r}_{s_0, \text{sloc}}}.
$$
By Proposition \ref{interpolation}, 
$$
\big\|m(L)-m_0(L)\big\|_{L^{p}(X) \to L^{p}(X)} \lesssim_{p,s,r} \|m\|_{L^{r}_{s, \text{sloc}}}.
$$

By duality, we can extend from $1<p<2$ to $1<p<\infty$. For $t,\tau\in \mathbb{R}$, by a similar argument as in Proof of Theorem \ref{Sobolev-bound},
$$
\big\|m(L) \big\|_{NL^p_{2\kappa \tau}(X) \to NL^p_{2\kappa \tau-2\kappa t}(X)} \lesssim_{t, \tau, p, s,r} |m(0)| + \Big\| \frac{m(\lambda)}{\Lambda^t} \Big\|_{L^r_{s, \text{sloc}}},
$$
for  
$$
s > Q\Big|\frac{1}{p}-\frac{1}{2}\Big|, \quad 0\leq \frac{1}{r} < \min \Big\{1, s-Q\Big|\frac{1}{p}-\frac{1}{2}\Big| \Big\}.
$$
\end{proof}

\appendix

\section{An interpolation result}\label{B}
The Calder\'on-Torchinsky multiplier theorem \cite{torchinsky} (see also Chapter 5 of Grafakos \cite{grafakosbook}) can be adapted into the following interpolation result. Recall that the smooth bump functions $\phi_j$ are defined in (\ref{phi}), and that $\lambda_0$ is the smallest nonzero spectrum of $L$.

\begin{proposition}\label{interpolation}
Let $1<r_0, r_1\leq \infty$, $1< p_0, p_1<\infty, 0<s_0, s_1<\infty, r_0s_0>1, r_1s_1>1$. Suppose for every $m$ on $[0,\infty)$ supported away from $0$,
$$
\big\|m(L)\big\|_{L^{p_0}(X) \to L^{p_0}(X)} \lesssim \|m\|_{L^{r_0}_{s_0, \text{sloc}}}, \quad \big\|m(L)\big\|_{L^{p_1}(X) \to L^{p_1}(X)} \lesssim \|m\|_{L^{r_1}_{s_1, \text{sloc}}}.
$$
For $0<\theta<1$, let
$$
\frac{1}{p} = \frac{1-\theta}{p_0} + \frac{\theta}{p_1}, \quad \frac{1}{r} = \frac{1-\theta}{r_0} + \frac{\theta}{r_1}, \quad s> (1-\theta)s_0 + \theta s_1.
$$
Then for every $m$ on $[0,\infty)$ supported away from $0$,
$$
\big\|m(L)\big\|_{L^p(X) \to L^p(X)} \lesssim_{\theta,s_0, s_1,s} \|m\|_{L^r_{s, \text{sloc}}}.
$$
\end{proposition}

For $z= u+iv\in \mathbb{C}$, denote the Bessel potential operator of order $z$ as 
$$
(I+\Delta)^{z/2} f:= \mathcal{F}^{-1}\big( (1+4\pi^2 |\xi|^2)^{z/2}\hat f(\xi) \big), \quad \forall f\in \mathcal{S}(\mathbb{R}),
$$
where $\mathcal{F}^{-1}$ denotes the inverse Fourier transform. For $1<r\leq\infty, 0<s<\infty$,
$$
\|f\|_{L^r_s(\mathbb{R})} =\|(I+\Delta)^{s/2}f\|_{L^r(\mathbb{R})}.
$$
By Lemma 4.2 in Calder\'on and Torchinsky \cite{torchinsky}, if $u>0$, the operator $(I+\Delta)^{-z/2}$ is the convolution with an $L^1$ function, and
$$
\|(I+\Delta)^{-z/2}\|_{L^r(\mathbb{R}) \to L^r(\mathbb{R})}  \lesssim (1+|z|)^2 \frac{1+u}{u} \lesssim_u 1+v^2, \quad \forall 1<r\leq \infty.
$$

\begin{proof}[Proof of Proposition \ref{interpolation}]
Assume $m$ is supported on $[\lambda_0/2, \infty)$ with $\|m\|_{L^r_{s,\text{sloc}}}<\infty$. Let
$$
s_0' = s_0 + s- \big((1-\theta)s_0 + \theta s_1 \big), \quad s_1' = s_1 + s- \big((1-\theta)s_0 + \theta s_1 \big).
$$
Denote
$$
\text{sgn}\,z:= e^{i\text{Arg}\,z}, \quad m_j(\lambda) := m(2^{2\kappa j}\lambda) \phi_j(\lambda), \quad \psi_j(\lambda) := \sum_{\substack{l\in \mathbb{N}\\ |l-j|\leq 1}} \phi_l(2^{-2\kappa l}\lambda), \quad \forall j\in \mathbb{N}.
$$
Denote the strip $S:= \{z\in \mathbb{C}: 0< \text{Re}\,z < 1\}$. For $z\in \bar S$, define
\begin{equation*}
m_z(\lambda) = \sum_{j\in \mathbb{N}} \Big((I+\Delta)^{\frac{-(1-z)s_0'-zs_1'}{2}} \big(|(I+\Delta)^{s/2}m_j|^{(1-z)\frac{r}{r_0}+ z \frac{r}{r_1}} \,\text{sgn}\,((I+\Delta)^{s/2}m_j)\big) \Big)(2^{-2\kappa j}\lambda) \cdot \psi_j(\lambda).
\end{equation*}
Note $m_\theta = m$.

We first show $m_z\in L^\infty$. If $r=\infty$, then $r_0=r_1=\infty$, and
\begin{align*}
&\quad \Big\|(I+\Delta)^{\frac{-(1-z)s_0'-zs_1'}{2}} \big(|(I+\Delta)^{s/2}m_j|^{(1-z)\frac{r}{r_0}+ z \frac{r}{r_1}} \,\text{sgn}\,((I+\Delta)^{s/2}m_j)\big) \Big\|_{L^\infty(\mathbb{R})}\\
&=\Big\|(I+\Delta)^{\frac{-(1-z)s_0'-zs_1'}{2}} \big((I+\Delta)^{s/2}m_j\big) \Big\|_{L^\infty(\mathbb{R})} \lesssim_{z, s_0', s_1'} \big\|(I+\Delta)^{s/2}m_j\big\|_{L^\infty(\mathbb{R})} \leq \|m\|_{L^\infty_{s, \text{sloc}}}<\infty.
\end{align*}
If $1<r<\infty$, for $z= u+ iv\in \bar S$, denote 
$$
\frac{1}{r_u}: = \frac{1-u}{r_0} + \frac{u}{r_1}, \quad s_u: = (1-u)s_0 + u s_1, \quad s_u': = (1-u)s_0' + u s_1'.
$$
By the Sobolev embedding theorem,
\begin{align*}
&\quad \Big\|(I+\Delta)^{\frac{-(1-z)s_0'-zs_1'}{2}} \big(|(I+\Delta)^{s/2}m_j|^{(1-z)\frac{r}{r_0}+ z \frac{r}{r_1}} \,\text{sgn}\,((I+\Delta)^{s/2}m_j)\big) \Big\|_{L^\infty(\mathbb{R})}\\
&\lesssim_{s_u, r_u} \Big\|(I+\Delta)^{\frac{-(1-z)s_0'-zs_1'}{2}} \big(|(I+\Delta)^{s/2}m_j|^{(1-z)\frac{r}{r_0}+ z \frac{r}{r_1}} \,\text{sgn}\,((I+\Delta)^{s/2}m_j)\big) \Big\|_{L^{r_u}_{s_u}(\mathbb{R})}\\
&\lesssim_{z, s_0, s_0', s_1, s_1'} \Big\||(I+\Delta)^{s/2}m_j|^{(1-z)\frac{r}{r_0}+ z \frac{r}{r_1}} \,\text{sgn}\,((I+\Delta)^{s/2}m_j)\Big\|_{L^{r_u}(\mathbb{R})} =\big\|(I+\Delta)^{s/2}m_j\big\|_{L^r(\mathbb{R})}^{\frac{r}{r_u}} \leq \|m\|_{L^r_{s, \text{sloc}}}^{\frac{r}{r_u}}<\infty.
\end{align*}
Hence $\|m_z\|_{L^\infty(\mathbb{R})}<\infty$.

By the definition of $m_z(\lambda)$, $m_{z}(2^{2\kappa k} \lambda) \phi(\lambda)$ contains at most five nonzero terms for any $k\in \mathbb{Z}$, and $m_z(2^{2\kappa k} \lambda)\phi(\lambda)\equiv 0$ for sufficiently small $k\in \mathbb{Z}$. Thus to bound $\big\| m_{z}(2^{2\kappa k} \lambda) \phi(\lambda) \big\|_{L^{r_u}_{s_u}(\mathbb{R})}$, it suffices to bound the following:
\begin{equation}\label{last}
\begin{aligned}
&\quad  \Big\|(I+\Delta)^{\frac{-(1-z)s_0'-z s_1'}{2}}\big(|(I+\Delta)^{s/2}m_j|^{(1-z)\frac{r}{r_0}+ z \frac{r}{r_1}} \,\text{sgn}\,((I+\Delta)^{s/2}m_j)\big)\cdot \phi \Big \|_{L^{r_u}_{s_u}(\mathbb{R})}\\
&\lesssim_{s_0, s_1, s, \theta} \Big\|(I+\Delta)^{\frac{-(1-z)s_0'-z s_1'}{2}}\big(|(I+\Delta)^{s/2}m_j|^{(1-z)\frac{r}{r_0}+ z \frac{r}{r_1}} \,\text{sgn}\,((I+\Delta)^{s/2}m_j)\big) \Big \|_{L^{r_u}_{\frac{s_u+s_u'}{2}}(\mathbb{R})}\\
&=\Big\|(I+\Delta)^{\frac{s_u+s_u'}{4}} (I+\Delta)^{\frac{-(1-z)s_0'-z s_1'}{2}}\big(|(I+\Delta)^{s/2}m_j|^{(1-z)\frac{r}{r_0}+ z \frac{r}{r_1}} \,\text{sgn}\,((I+\Delta)^{s/2}m_j)\big)\Big\|_{L^{r_u}(\mathbb{R})} \\
&\lesssim_{\theta, s_0,s_1, s} (1+v^2) \cdot  \big\||(I+\Delta)^{s/2}m_j|^{(1-z)\frac{r}{r_0}+ z \frac{r}{r_1}} \,\text{sgn}\,((I+\Delta)^{s/2}m_j)\big\|_{L^{r_u}(\mathbb{R})} \leq (1+v^2) \|m\|_{L^r_{s, \text{sloc}}}^{\frac{r}{r_u}}. 
\end{aligned}
\end{equation}
Therefore 
$$
\|m_{z}\|_{L^{r_u}_{s_u, \text{sloc}}} \lesssim_{\theta, s_0,s_1,s} (1+v^2) \|m\|_{L^r_{s, \text{sloc}}}^{\frac{r}{r_u}} <\infty.
$$
By the Sobolev embedding theorem, $m_z$ is continuous on $(0,\infty)$. Hence $m_z(L)$ is defined.

Fix arbitrary $f,g\in C^\infty(X)$, we need to show
\begin{align*}
F(z):&= \langle m_z(L)f, g \rangle_{L^2(X)} =  \int_0^\infty m_z(\lambda)\,d\big\langle E_L(\lambda) f, g\big\rangle 
\end{align*}
is holomorphic on $S$. Note $d\langle E_L(\lambda) f, g\rangle$ is a Borel measure. For each $j, N\in \mathbb{N}$, define
$$
d\mu_{N,j}(\lambda) = \big(1-\phi_0(2^{2\kappa j+2\kappa N} \lambda) \big)\cdot \phi_0(2^{2\kappa j-2\kappa N} \lambda) \cdot \Big(N\phi_0\big(N\cdot \big) * d\big\langle E_L( \cdot) f, g\big\rangle\Big)(2^{2\kappa j} \lambda) \cdot \psi_j(2^{2\kappa j}\lambda).
$$
Note each $\mu_{N,j}$ has compact support in $(0,\infty)$ and with smooth density. Define
\begin{align*}
F_N(z)&:= \int_0^\infty m_z(\lambda) (1-\phi_0(2^{2\kappa N}\lambda)) \phi_0(2^{-2\kappa N}\lambda) \big(N\phi_0(N\cdot) * d\langle E_L(\cdot)f, g\rangle \big)(\lambda)\\
&=\sum_{j\in \mathbb{N}} \int_0^\infty \Big( (I+\Delta)^{\frac{-(1-z)s_0'-zs_1'}{2}} \big(|(I+\Delta)^{s/2}m_j|^{(1-z)\frac{r}{r_0}+ z \frac{r}{r_1}} \,\text{sgn}\,((I+\Delta)^{s/2}m_j)\big)\Big)(\lambda)\,d\mu_{N,j}(\lambda).   
\end{align*}
Note each $F_N$ only has finitely many nonzero terms. Denote by $\bar z$ the complex conjugate of $z$. Easy to verify that each term is equal to
$$
\int \int \frac{\big(|(I+\Delta)^{s/2}m_j|^{(1-z)\frac{r}{r_0}+ z \frac{r}{r_1}} \,\text{sgn}\,((I+\Delta)^{s/2}m_j)\big)(\lambda)}{1+4\pi^2 \lambda^2} \cdot (I-\partial_\xi^2)\Big((1+4\pi^2 |\xi|^2)^{\frac{-(1-z)s_0'- zs_1'}{2}} \check \mu_{N,j}(\xi)\Big) e^{-2\pi i \lambda \xi}\,d\xi\,d\lambda,
$$
which is holomorphic in $S$ and continuous on $\bar S$. Note $|(I+\Delta)^{s/2}m_j|^{(1-z)\frac{r}{r_0}+ z \frac{r}{r_1}} \,\text{sgn}\,((I+\Delta)^{s/2}m_j)\in L^{r_u}(\mathbb{R})$ due to the last inequality in (\ref{last}), and thus 
$$
\frac{\big(|(I+\Delta)^{s/2}m_j|^{(1-z)\frac{r}{r_0}+ z \frac{r}{r_1}} \,\text{sgn}\,((I+\Delta)^{s/2}m_j)\big)(\lambda)}{1+4\pi^2 \lambda^2}\in L^1(\mathbb{R}).
$$ 
Since $F_N$ converges to $F$ locally uniformly in $\bar S$, $F$ is also holomorphic in $S$ and continuous on $\bar S$. Note $|F|$ has at most polynomial growth with respect to $\text{Im}\,z$ on $\bar S$.

We have
\begin{align*}
&\big\|m_{iv}(L)\big\|_{L^{p_0}(X) \to L^{p_0}(X)} \lesssim_{\theta, s_0,s_1,s} (1+v^2) \|m\|_{L^r_{s, \text{sloc}}}^{\frac{r}{r_0}},\\
&\big\|m_{1+iv}(L)\big\|_{L^{p_1}(X) \to L^{p_1}(X)} \lesssim_{\theta, s_0,s_1,s} (1+v^2) \|m\|_{L^r_{s, \text{sloc}}}^{\frac{r}{r_1}}.    
\end{align*}
Therefore by Stein's interpolation of analytic families of operators,
$$
\big\|m(L)\big\|_{L^{p}(X) \to L^{p}(X)} \lesssim_{\theta, s_0, s_1, s} \|m\|_{L^r_{s, \text{sloc}}}.
$$
\end{proof}

\bibliographystyle{amsalpha}
\bibliography{ref}

\vspace{2em}
\noindent
Lingxiao Zhang\\
Department of Mathematics\\
University of Connecticut\\
341 Mansfield Rd, Storrs, CT, 06269, USA\\
email: \href{mailto:lingxiao.zhang@uconn.edu}{lingxiao.zhang@uconn.edu}\\
MSC: primary 42B15, 43A85; secondary 42B35

\end{document}